\documentclass[12pt,a4paper]{amsart}

\usepackage{hyperref}
\usepackage{color}

\usepackage{amsmath,amsfonts,amssymb,xypic}

\theoremstyle{plain}
\newtheorem*{theorem*}{Theorem}
\newtheorem*{lemma*} {Lemma}
\newtheorem*{corollary*} {Corollary}
\newtheorem*{proposition*} {Proposition}
\newtheorem*{conjecture*}{Conjecture}

\newtheorem{theorem}{Theorem}[section]
\newtheorem{lemma}[theorem]{Lemma}
\newtheorem{corollary}[theorem]{Corollary}

\newtheorem{maintheorem}[theorem]{Main Theorem}

\theoremstyle{definition}

\theoremstyle{remark}
\newtheorem*{remark}{Remark}

\theoremstyle{definition}

\def\be{\begin{equation}} \def\ee{\end{equation}}

\def\Q{\Bbb{Q}}

\def\Z{\Bbb{Z}}

\def\ll{\langle}
\def\rr{\rangle}

\def\sm{\setminus}
\def\bp{\begin{pmatrix}}
\def\ep{\end{pmatrix}}
\def\bn{\begin{enumerate}}
\def\en{\end{enumerate}}

\def\Hom{\operatorname{Hom}}

\def\ba{\begin{array}}
\def\ea{\end{array}}

\def\d{\delta}

\def\lk{\operatorname{lk}}

\def\fr12{\frac{1}{2}}

\def\ol{\overline}

\def\coker{\operatorname{coker}}
\def\im{\operatorname{im}}

\def\S{\Sigma}

\def\zt{\Z[t^{\pm 1}]}

\def\Kh{\operatorname{Kh}}
\def\Id{\operatorname{Id}}

\def\toiso{\xrightarrow{\simeq}}
\def\wt{\widetilde}

\begin{document}

\title{Cobordisms to weakly splittable links}


\author{Stefan Friedl}
\address{Mathematisches Institut\\ Universit\"at zu K\"oln\\   Germany}
\email{sfriedl@gmail.com}
\author{Mark Powell}
\address{Department of Mathematics\\ Indiana University Bloomington\\ IN, USA}
\email{macp@indiana.edu}

\def\subjclassname{\textup{2010} Mathematics Subject Classification}
\expandafter\let\csname subjclassname@1991\endcsname=\subjclassname
\expandafter\let\csname subjclassname@2000\endcsname=\subjclassname


\begin{abstract}
We show that if a link $L$  with non--zero Alexander polynomial admits a locally flat cobordism to a `weakly $m$--split link', then the cobordism must have genus at least $\lfloor \frac{m}{2}\rfloor$.  This generalises a recent result of J.~Pardon.
\end{abstract}

\maketitle

\section{Introduction}

A \emph{$k$--component link} $L$ is an oriented $k$--component one dimensional submanifold of $S^3$
with an ordering of the components. A smooth (respectively topological) \emph{cobordism} $C$ between two $k$--component links $L=L_1\,\sqcup \dots \sqcup \, L_k$ and $J = J_1\,\sqcup \dots \sqcup \, J_k$ is a
smooth (respectively locally flat) oriented surface $C=C_1\,\sqcup\dots \sqcup \, C_k\subset S^3\times [0,1]$, such that $C_i\,\cap\, S^3\times \{0\}=L_i\times \{0\}$ and $C_i\,\cap\, S^3\times \{1\}=-J_i\times \{1\}$ for $i=1,\dots,k$.

J.~Pardon \cite{Pa11} recently extended Rasmussen's knot concordance invariant \cite{Ra10} to a concordance invariant for links.
As one of the main applications of the new invariant Pardon proves the following theorem:

\begin{theorem} \label{thm:pardon} A smooth cobordism between a $\Kh$--thin link and a
link split into $m$ components must have genus at least $\lfloor \frac{m}{2}\rfloor$.
\end{theorem}

We will not give a precise definition of $\Kh$--thin links. Roughly speaking a link is $\Kh$--thin if its reduced Khovanov homology (see \cite{Kh00})
is concentrated on a diagonal.
Alternatively one can define $\Kh$--thin links using ordinary Khovanov homology (see \cite[Proposition~4]{Kh03}).  What is of interest to us is that we have the following inclusions:
\[ \ba{rclcl} \left\{\ba{c}\mbox{non-split}\\\mbox{alternating}\\\mbox{links}\ea\right\}
&\subset&
\left\{\ba{c}\mbox{quasi-alternating}\\\mbox{links}\ea\right\}
&\subset &
\left\{\ba{c}\mbox{$\Kh$--thin}\\\mbox{links}\ea\right\}\\[2mm]
&\subset&
\left\{\ba{c}\mbox{links with}\\\mbox{$\det(L)\ne 0$}\ea\right\}
&\subset&
\left\{\ba{c}\mbox{links with}\\\mbox{$\Delta_L(t)\ne 0$}\ea\right\}.\ea \]
Here we denote the determinant of a link $L$ by $\det(L)$, and the one--variable Alexander polynomial by $\Delta_L(t)$.
We refer to \cite{MO08,OS05} for details on the first two inclusions. For the third inclusion see e.g. \cite[Proposition~2.5]{Wa08}.
The last inclusion follows from the equality $\det(L)=|\Delta_L(-1)|$.

As indicated, it is known that these inclusions are proper inclusions: see e.g. \cite{Gr10} for the second inclusion.

Pardon asks whether the Alexander module of a link can be used to reprove  Theorem \ref{thm:pardon}
for certain classes of links (e.g. quasi--alternating links).
In this note we follow through on Pardon's idea, and both reprove and extend
 Theorem \ref{thm:pardon}.

In the following, we say that a link $J = J_1\,\sqcup \dots \sqcup\, J_k$ is \emph{weakly $m$--split} if there exist $m$ disjoint oriented surfaces $\S_1,\dots,\S_m$ with non--trivial boundary, embedded in $S^3$,
such that $\partial \S_1\,\sqcup \dots \sqcup\, \partial \S_m=J$.  Note that a link which splits into $m$ sublinks is in particular weakly $m$--split.
On the other hand $m$--component boundary links, which in general do not split, are nevertheless also weakly $m$--split.

Our main technical result (see Theorem \ref{mainthm2}) relates the ranks of Alexander modules of links to the genera of cobordisms
between them. This will allow us to prove the following generalisation of Pardon's theorem:

\begin{maintheorem}\label{mainthm}
Let $L$ be a link with non--zero Alexander polynomial $\Delta_L\in \zt$.  A topological cobordism between $L$ and a
weakly $m$--split link $J$ must have genus at least $\lfloor \frac{m}{2}\rfloor$.
\end{maintheorem}

\subsection*{Acknowledgements}

MP would like to thank Charles Livingston and Kent Orr for helpful discussions.
We also wish to thank the referee for carefully reading our paper.

\section{Proof of the main theorem}\label{Section:proof_main_thm}

\subsection{Preliminaries on Alexander modules}

Let $L$ be a $k$-component link. We write $Y_L:=S^3\sm \nu L$.
Given a $k$--component link $L$ and a homomorphism $\phi\colon \Z^k\to H$ to a free abelian group, we consider the coefficient system corresponding to
\[\pi_1(Y_L)\to H_1(Y_L;\Z) \cong \Z^k \xrightarrow{\phi} H \to \Z[H] \to \Q(H),\]
where the second map is the canonical isomorphism sending the $i$--th oriented meridian to $e_i$, the $i$--th standard basis element of $\Z^k$,
and where $\Q(H)$ denotes the quotient field of the group ring $\Z[H]$.  We define:
\[r(L,\phi):=\dim_{\Q(H)}H_1(Y_L;\Q(H)).\]

We denote by $\d:\Z^k\to \Z$ the `diagonal' homomorphism defined by $\d(e_i)=1$, $i=1,\dots,k$.
Note that the corresponding $\pi_1(Y_L)\to \Z$ is the unique epimorphism which sends each oriented meridian to $1$.
We write:
\[r(L):=r(L,\d:\Z^k\to \Z).\]
As usual we identify the group ring $\Z[\Z]$ with $\zt$ and we denote by  $H_1(Y_L;\zt)$ the  Alexander module corresponding to the  canonical epimorphism.
Furthermore we denote the order of this Alexander module by $\Delta_L=\Delta_L(t)\in \zt$.
Note that $r(L)=0$ if and only if $\Delta_L\ne 0$.

Our proof of Theorem \ref{mainthm} hinges on the following bound on the rank of the twisted homology of a weakly $m$--split link.

\begin{lemma}\label{lemma:H_1_weakly_m_split}
Let $J$ be a weakly $m$--split link. Then
\[ r(J) \geq m-1.\]
\end{lemma}

The statement of the lemma is well--known for boundary links. The standard proof for boundary links easily generalises to the case of weakly $m$--split links. For the reader's convenience we provide the details.

\begin{proof}
Since $J$ is a weakly $m$--split link, there exist $m$ disjoint oriented surfaces $\S_1,\dots,\S_m$ with non--trivial boundary,
such that $\partial \S_1\,\sqcup \dots \sqcup\, \partial \S_m=J$. In the following, by slight abuse of notation, we also denote the intersection of the surface $\S_i$ with $Y_J$ by $\S_i$, for $i=1,\dots,m$.  Note that $\S:=\S_1\,\sqcup \dots \sqcup\, \S_m$ is
dual to the canonical homomorphism $\phi\colon \pi_1(Y_J)\to \Z$ under the isomorphisms
\[H_2(Y_J,\partial Y_J;\Z) \toiso H^1(Y_J;\Z) \toiso \Hom_{\Z}(H_1(Y_J;\Z),\Z) \toiso \Hom(\pi_1(Y_J),\Z),\]
since both homomorphisms to $\Z$ send each oriented meridian to one.
In particular $\S$ is non--separating.    Since $\Q(t)$ is flat as a module over $\zt$ we have
\[H_*(Y_J;\Q(t)) \cong H_*(Y_J;\zt) \otimes_{\zt}\Q(t) \cong H_*(\wt{Y}_J;\Z) \otimes_{\zt} \Q(t),\]
where $\wt{Y}_J$ is the infinite cyclic cover of $Y_J$, constructed by cutting $Y_J$ along $\S$ to obtain $Y^\S$ and then gluing the fundamental domains $\{t^iY^\S\,|\, i \in \Z\}$ together along $\{t^i\S\,|\,i \in \Z\}$.  The Mayer-Vietoris sequence of
\[\wt{Y}_J = \bigsqcup_{i \text{ odd}}\, t^iY^\S \cup_{\bigsqcup_{i \in \Z}t^i\S} \bigsqcup_{i \text{ even}} t^iY^\S \]
is as follows:
\[ \ba{ccccccc} &&&&&H_1(\wt{Y}_J;\Z)&\\
\to& \bigoplus_{i\in\Z} H_0(t^i\S;\Z)&\to& \bigoplus_{i\in\Z} H_0(t^i Y^\S;\Z)&\to& H_0(\wt{Y}_J;\Z).\ea\]
Considering all of the homology groups as $\zt$ modules, this is equivalent to:
 \begin{align*} H_1(Y_J;\zt)&\\
\to H_0(\S;\Z) \otimes_{\Z} \zt\to H_0(Y^\S;\Z)\otimes_{\Z} \zt \to H_0(Y_J;\zt)&.\end{align*}
Tensoring with $\Q(t)$, considered as a $\zt$-module, yields:
\[ \ba{ccccccc} &&&&&H_1(Y_J;\Q(t))&\\
\to& H_0(\S;\Z) \otimes_{\Z} \Q(t)&\to& H_0(Y^\S;\Z)\otimes_{\Z} \Q(t) &\to& H_0(Y_J;\Q(t)).\ea\]
Since each $\S_i$ is connected, $H_0(\S;\Z) \cong \bigoplus_{i=1}^m\,H_0(\S_i;\Z) \cong \Z^m$, so that $H_0(\S;\Z) \otimes_{\Z} \Q(t) \cong \Q(t)^m$.  Similarly, since $Y^\S$ is connected, $H_0(Y^\S;\Z)\otimes_{\Z} \Q(t) \cong \Q(t)$.  The coefficient system $\pi_1(Y_J) \to \Z=\ll t\rr$ is non-trivial, which implies that $$H_0(Y_J;\Q(t)) \cong \coker((t-1) \colon \Q(t) \to \Q(t)) \cong 0.$$  Therefore, by exactness,
\begin{align*}
  &\ker(H_0(\S;\Z) \otimes_{\Z} \Q(t)\to H_0(Y^\S;\Z)\otimes_{\Z} \Q(t)) \cong \Q(t)^{m-1}\\ \cong & \im(H_1(Y_J;\Q(t))\to H_0(\S;\Z) \otimes_{\Z} \Q(t)),
\end{align*}
so that $\dim_{\Q(t)}(H_1(Y_J;\Q(t))) \geq m-1$ as claimed.
\end{proof}

\subsection{Cobordisms and ranks of  Alexander modules}

We start our discussion of cobordisms between links with the following elementary lemma:

\begin{lemma}
Let $L$ and $J$ be two $k$--component links. Then the following are equivalent:
\bn
\item there exists a smooth cobordism between $L$ and $J$;
\item there exists a topological cobordism between $L$ and $J$;
\item we have  $\lk(L_i,L_j)=\lk(J_i,J_j)$ for all $i,j=1,\dots,k$.
\en
\end{lemma}

\begin{proof}
Any smooth cobordism is also a topological cobordism, so (1) implies (2). The fact that  (2) implies (3) follows from the definition of linking numbers in terms of surfaces in the 4--ball.
Finally, assume that we have $\lk(L_i,L_j)=\lk(J_i,J_j)$ for all $i,j=1,\dots,k$.
Since $H_1(S^3\times [0,1];\Z)=0$, there exist oriented, smoothly embedded surfaces $F_1,\dots,F_k$ with $\partial F_i=L_i\cup -J_i$, $i=1,\dots,k$. We can furthermore assume that the surfaces are in general position. It follows from the definition of linking numbers that the signed count of double points satisfies:
\[ F_i\cdot F_j=\lk(L_i,L_j)-\lk(J_i,J_j)=0.\]
We can therefore pair up intersection points of $F_i$ and $F_j$.  Let $p,q \in F_i \cap F_j$ have opposite signs, and let $\gamma$ be a path in $F_j$ from $p$ to $q$. Remove an open disc neighbourhood of $p$ and of $q$ from $F_i$, add a tube $\gamma \times S^1$ to $F_i\setminus(\nu p\, \sqcup\, \nu q)$ and smooth the corners.  By repeating this operation we can arrange that the surfaces $F_1,\dots,F_k$ are disjoint, which implies (1).

\end{proof}

Given two links $L$ and $J$ with the same number of components we define, for $CAT = top,\, smooth$:
\[ \ba{rcl} g^{CAT}(L,J)&\hspace{-0.2cm}:=&\hspace{-0.2cm}\min\{ g(C)\,|\, C\mbox{ a } CAT \mbox{ cobordism between }L \mbox{ and }J\},\ea\]
where $g(C) = \frac{1}{2}(\beta_1(C) - \beta_0(C))$ is the genus of $C$.  If no cobordism between $L$ and $J$ exists, then we define $g^{CAT}(L,J):=\infty$.

We need one last definition before we can state our main theorem.
We say that a homomorphism $\phi\colon \Z^k\to H$ to a free abelian group $H$ is \emph{admissible} if $\phi(e_i)\ne 0$ for $i=1,\dots,k$, where $e_i$ is the $i$--th basis element of $\Z^k$.

Our main technical theorem is then as follows:

\begin{theorem}\label{mainthm2}
Let $L$ and $J$ be $k$--component links, let $H$ be a free abelian group and let $\phi \colon \Z^k \to H$ be admissible. Then
\[ g^{top}(L,J) \geq \frac{1}{2}\left|r(J,\phi)-r(L,\phi)\right|.\]
\end{theorem}

\begin{proof}[Proof of Theorem \ref{mainthm} assuming Theorem \ref{mainthm2}]
Let $L$ be a link with non--zero Alexander polynomial $\Delta_L\in \zt$ and let $C$ be a topological cobordism to a
weakly $m$--split link $J$.
Recall that a link $L$ has non--zero single variable Alexander polynomial $\Delta_L(t)$ if and only if $r(L)=0$.
We apply Theorem \ref{mainthm2} to the admissible `diagonal' homomorphism $\delta \colon \Z^k \to \Z$. We deduce that $g(C) \geq \frac{1}{2}r(J)$, so that $g(C) \geq \frac{m-1}{2}$ by Lemma \ref{lemma:H_1_weakly_m_split}.
\end{proof}

\begin{remark}
\bn
\item  Kawauchi \cite[Theorem~A]{Ka78} showed that if $L$ and $J$ are in fact \emph{concordant}, i.e. cobordant via annuli,
 then $r(L,\Id)=r(J,\Id)$.  Kawauchi's proof is easily modified to show that his result holds for any admissible $\phi$.
 Our theorem can therefore be thought of as a generalisation of Kawauchi's result to the case $r(L,\phi)\ne r(J,\phi)$.
\item It is straightforward to verify that our proof generalises to the consideration
of links $L, J$ in integral homology 3--spheres $M, N$ which are cobordant via a cobordism in
an integral homology $S^3\times [0,1]$ with boundary $M \cup -N$.
\item
 Let $L$ and $J$ be two $k$--component links.
 The Gordian distance $d(L,J)$ is defined as the minimal number of crossing changes needed to turn $L$ into $J$.  The \emph{intra--component} Gordian distance $d^{comp}(L,J)$ is defined to be the minimal number of crossing changes which involve the same component of the link, that are needed to transform $L$ to $J$.  If no such sequence of moves exists, then we say that $d^{comp}(L,J) = \infty$.
By a standard trick which replaces a neighbourhood of a double point of an immersed surface with a twisted annulus, one sees that $d^{comp}(L,J)\geq \,g^{smooth}(L,J)$.

Given a link $L$ we denote by $m(L)$ the minimal number of generators of the Alexander module $H_1(Y_L;\zt)$.
Kawauchi \cite[Theorem~2.3]{Ka96} showed that
 \[ d(L,J)\geq |m(L)-m(J)|.\]
Theorem \ref{mainthm2} and Kawauchi's result now fit into the following diagram
\[ \xymatrix{ d^{comp}(L,J) \ar[d]\ar[r] & g^{smooth}(L,J)\ar[r] & g^{top}(L,J) \ar[d]  \\
d(L,J) \ar[r] & |m(L)-m(J)| \ar[r]& \frac{1}{2}|r(L)-r(J)|}\]
where an arrow $A\to B$ indicates that $B$ is a lower bound on $A$.  Our result is therefore seen to be independent of Kawauchi's.
 \en
\end{remark}

\subsection{Proof of Theorem \ref{mainthm2}}

In this section we will give a proof of Theorem \ref{mainthm2}.
Throughout, we let
 $C$ be a topological cobordism between two $k$--component links $L$ and $J$ and we denote the genus of $C$ by $g$.
We write $$X_C:=(S^3\times [0,1])\sm \nu C.$$
We now have the following lemma relating the integral homology groups of $X_C$, $Y_L$, $Y_J$, $C$, $L$ and $J$.

\begin{lemma}\label{lem:Zhomology}
The integral homology of $X_C$ is given by:
\[H_i(X_C;\Z) \cong  \begin{cases}
  \Z & i=0; \\ \Z^k & i=1; \\ \Z^{2g+k-1} & i=2; \\ 0 & i \geq 3.
\end{cases}\]
Furthermore the inclusion induced maps $H_1(Y_L;\Z)\to H_1(X_C;\Z)$ and $H_1(Y_{J};\Z)\to H_1(X_C;\Z)$
are isomorphisms, such that the image of the $i$--th meridian of $L$ in $H_1(X_C;\Z)$ agrees with the image of the $i$--th meridian of $J$.
\end{lemma}

This lemma can be seen as a variation on Alexander duality in a ball.
The statement is well-known to the experts, but we give a proof for the reader's convenience.

\begin{proof}
In this proof, all homology groups are with $\Z$--coefficients.  We therefore allow ourselves to omit the coefficients from the notation.
In the following we identify $L$ with $L\times \{0\} \subset S^3 \times [0,1]$ and $J$ with $J\times \{1\}$.
We will also write $S_L=S^3\times \{0\}$ and $S_J=S^3\times \{1\}$.
Finally we write $\eta L = L\times D^2 \subset S_L$, $\eta J=J\times D^2 \subset S_J$ and $\eta C=C\times D^2 \subset S^3 \times [0,1]$.

We first consider the following commutative diagram,
where the horizontal isomorphisms are given by Poincar\'e duality and excision:
\begin{equation}\label{eqn:comm_diagram1}
\ba{cccccccc}
H_i(Y_L)&\cong& H^{3-i}(Y_L,\partial Y_L)&\cong&H^{3-i}(S_L,\eta L)\\
\downarrow&&\downarrow&&\downarrow\\
H_i(X_C)&\cong &H^{4-i}(X_C,\partial X_C)&\cong&H^{4-i}(S^3\times I,
S_L \cup \eta C\cup S_J).\ea\end{equation}
The first vertical map is induced by inclusion.  The other vertical maps in the diagram above, and indeed for all but one vertical map in the next two diagrams, are given by a composition of excision and maps from long exact sequences of certain pairs.  For example, the second vertical map is given by:
\[H^{3-i}(Y_L,\partial Y_L) \cong H^{3-i}(\partial X_C, Y_J\, \cup\, C \times S^1) \to H^{3-i}(\partial X_C) \to H^{4-i}(X_C,\partial X_C).\]

Next, we have the commutative diagram:
 \begin{equation}\label{eqn:comm_diagram2}
\ba{ccccccccc}
H^{2-i}(\eta L)&\hspace{-0.28cm}\to\hspace{-0.28cm} & H^{3-i}(S_L,\eta L)&\hspace{-0.28cm}\to\hspace{-0.28cm} & H^{3-i}(S_L) \\
\downarrow&&\downarrow&&\downarrow\\
H^{3-i}(\eta C,\eta L\hspace{-0.07cm}\cup\hspace{-0.07cm} \eta J)&\hspace{-0.28cm}\to\hspace{-0.28cm}&
H^{4-i}(S^3\hspace{-0.07cm}\times\hspace{-0.07cm} I,
S_L\hspace{-0.07cm}\cup\hspace{-0.07cm} \eta C\hspace{-0.07cm}\cup\hspace{-0.07cm} S_J)&\hspace{-0.28cm}\to\hspace{-0.28cm}& H^{4-i}(S^3\hspace{-0.07cm}\times\hspace{-0.07cm} I,S_L\hspace{-0.07cm}\cup\hspace{-0.07cm} S_J).\ea \end{equation}
Here, the top row is part of the long exact sequence in cohomology of the pair $(S^3,\eta L)$.  The bottom row belongs to the long exact sequence of the triple $$(S^3 \times I, S_L\, \cup\, \eta C \, \cup\, S_J,S_L\, \cup\, S_J),$$  noting that $$H^{3-i}(S_L\, \cup\, \eta C \, \cup\, S_J,S_L\, \cup\, S_J) \cong H^{3-i}(\eta C,\eta L\, \cup\, \eta J)$$
by excision.
Our last commutative diagram is as follows:
\begin{equation}\label{eqn:comm_diagram3}
\ba{cccccc}
H_{i-1}(L)&\cong & H^{2-i}(L)&\cong&H^{2-i}(\eta L)\\
\downarrow&&\downarrow&&\downarrow\\
H_{i-1}(C)&\cong &H^{3-i}(C,L\cup J)&\cong&H^{3-i}(\eta C,\eta L\cup \eta J).\ea \end{equation}
Here the first vertical map is induced by inclusion, while the other two vertical maps arise as described above.

Putting the  bottom rows of the diagrams (\ref{eqn:comm_diagram1}) and (\ref{eqn:comm_diagram3})
together with the long exact sequence in cohomology corresponding to the bottom row of (\ref{eqn:comm_diagram2}),
we obtain the following
long exact sequence:
\[ \ba{cccccc} &&&&& H^{3-i}(S^3\times I,S_L\cup S_J)\\
 \to & H_{i-1}(C)& \to& H_i(X_C) &\to& H^{4-i}(S^3\times I,S_L\cup S_J).\ea\]
Finally, note that the map
\[H^1(S^3 \times I, S_L\, \cup \,S_J) \to H^{1}(\eta C,\eta L\cup \eta J) \cong H_1(C)\]
sends the generator of $H^1(S^3 \times I, S_L\, \cup \,S_J)$ to an indivisible  element in $H_1(C)$,
namely the element corresponding to the sum of the oriented curves $C\cap S_L$.
We can now find the homology groups of $X_C$, by a straightforward calculation, to be as claimed.

When $i=1$, diagram (\ref{eqn:comm_diagram2}) extended one to the left becomes:
\begin{equation}\label{eqn:comm_diagram4}
\ba{ccccccccccc}
0&\hspace{-0.2cm}\to\hspace{-0.2cm} & H^{1}(\eta L)&\hspace{-0.2cm}\toiso\hspace{-0.2cm} & H^{2}(S_L,\eta L)&\hspace{-0.2cm}\to\hspace{-0.2cm} & 0\\
 &&\downarrow&&\downarrow\\
0 &\hspace{-0.2cm}\to\hspace{-0.2cm} &H^{2}(\eta C,\eta L\cup \eta J)&\hspace{-0.2cm}\toiso\hspace{-0.2cm}&
H^{3}(S^3\times I,
S_L\cup \eta C\cup S_J)&\hspace{-0.2cm}\to\hspace{-0.2cm}& 0.\ea \end{equation}
The second statement of the lemma then follows from combining the commutative diagrams (\ref{eqn:comm_diagram1}), (\ref{eqn:comm_diagram3}) and (\ref{eqn:comm_diagram4}), the fact that the inclusion map $L\to C$ induces an isomorphism on 0--th homology, and the observation that the r\^{o}les of $L$ and $J$ can be reversed, together with a careful consideration of orientations.
\end{proof}

\begin{corollary}\label{cor:euler_char}
The Euler characteristic of $X_C$ is $\chi^{\Q}(X_C) =  2g$.
\end{corollary}

From here on we fix an admissible homomorphism $\phi\colon \Z^k\to H$ to a free abelian group.  We write $l:=r(L,\phi)$ and $j:=r(J,\phi)$.  Without loss of generality we can assume that $j\geq l$.

Lemma \ref{lem:Zhomology} implies that the canonical isomorphisms $H_1(Y_L;\Z)\cong \Z^k$ and $H_1(Y_{J};\Z)\cong \Z^k$ extend uniquely to an isomorphism $H_1(X_C;\Z)\cong \Z^k$.
Therefore we can extend the coefficient systems corresponding to $\phi$ over $X_C$, and use the resulting coefficient system to define the twisted homology and cohomology of $X_C$.

Eventually the goal is to show that the difference of the dimensions of the $\Q(H)$--homologies of $L$ and $J$ gives a lower bound
on the second $\Q(H)$--homology of $X_C$, which in turn will give a lower bound on the genus of $C$.

With this in mind, we proceed to collect some facts about the homology of $X_C$ with $\Q(H)$ coefficients.
In order to do this we recall the following lemma (see e.g. \cite[Proposition~2.10]{COT03}).

\begin{lemma}\label{lem:cot}
Let $B\to A$ be an inclusion of connected spaces such that $H_i(A,B;\Z)=0$ for $i=0,\dots,r$.
Let $\pi_1(A)\to F$ be a homomorphism to a free abelian group $F$.  Then $H_i(A,B;\Q(F))=0$ for $i=0,\dots,r$.
\end{lemma}

The implications of Lemma \ref{lem:cot} which will be relevant for us are given in the next corollary.

\begin{corollary}\label{cor:chain_lifting_consequences}
We have:
\begin{enumerate}
  \item $H_1(X_C,Y_L;\Q(H)) \cong 0$; and
  \item $\dim_{\Q(H)} (H_1(X_C;\Q(H))) \leq l$.
\end{enumerate}
\end{corollary}

\begin{proof}
It follows from Lemma \ref{lem:Zhomology} combined with Lemma \ref{lem:cot}  that
$$H_i(X_C,Y_K;\Q(H)) \cong 0$$
 for $K=L,J$ and $i=0,1$.  With $K=L$, this proves (1).  Moreover, the long exact sequence of the pair \[H_1(Y_L;\Q(H)) \to H_1(X_C;\Q(H)) \to H_1(X_C, Y_L;\Q(H))\] shows us that the map $H_1(Y_L;\Q(H)) \to H_1(X_C;\Q(H))$ is surjective.  But $\dim_{\Q(H)} (H_1(Y_L;\Q(H))) =l$, which implies (2).
\end{proof}

\begin{lemma}\label{lem:Qt_homology}
The homology of $X_C$ twisted over $\Q(H)$ is given by:
\[H_i(X_C;\Q(H)) \cong 0\]
for $i \neq 1,2$.
\end{lemma}
\begin{proof}
Since the coefficient system arises from an admissible homomorphism $\phi \colon \Z^k \to H$, it is non-trivial: we therefore have  $H_0(X_C;\Q(H)) \cong 0$.

By the universal coefficient theorem and Poincar\'{e}-Lefschetz duality, we have isomorphisms:
\[H_4(X_C;\Q(H)) \cong H^4(X_C;\Q(H)) \cong \ol{H_0(X_C,\partial X_C;\Q(H))}.\]
The fact that $H_0(X_C;\Q(H)) \cong 0$ together with the long exact sequence of the pair $(X_C,\partial X_C)$, implies that $H_0(X_C,\partial X_C;\Q(H)) \cong 0$, so that as claimed $H_4(X_C;\Q(H)) \cong 0$.

Similarly, we have isomorphisms:
\[H_3(X_C;\Q(H)) \cong H^3(X_C;\Q(H)) \cong \ol{H_1(X_C,\partial X_C;\Q(H))}.\]
We now consider the long exact sequence of the triple $(X_C,\partial X_C, Y_L)$:
\[H_1(X_C,Y_L;\Q(H)) \to H_1(X_C,\partial X_C;\Q(H)) \to H_0(\partial X_C,Y_L;\Q(H)).\]
First, $H_0(\partial X_C;\Q(H)) \cong 0$, again since the coefficient system is non--trivial.  Thence, by the long exact sequence of the pair $(\partial X_C, Y_L)$, $H_0(\partial X_C,Y_L;\Q(H)) \cong 0$.

Moreover, by Corollary \ref{cor:chain_lifting_consequences} (1), $H_1(X_C,Y_L;\Q(H)) \cong 0$.  It follows that the middle homology group $H_1(X_C,\partial X_C;\Q(H)) \cong 0$, which in turn implies from the isomorphisms above that $H_3(X_C;\Q(H)) \cong 0$.  This completes the proof of Lemma \ref{lem:Qt_homology}.
\end{proof}

\begin{lemma}\label{lemma:homology_partial_X_C}
We have an isomorphism
  \[H_1(\partial X_C;\Q(H)) \cong H_1(Y_L;\Q(H)) \oplus H_1(Y_{J};\Q(H)).\]
\end{lemma}
\begin{proof}
  Note that
  \[\partial X_C = Y_L \cup_{\sqcup_k\, S^1 \times S^1} C \times S^1 \cup_{\sqcup_k\, S^1 \times S^1} -Y_{J}.\]
We claim that both $C \times S^1$ and $\sqcup_k\,S^1 \times S^1$ have vanishing $\Q(H)$-homology.  To see this, consider  each connected component of $C \times S^1$ and $\sqcup_k\, S^1 \times S^1$ separately.  Let $h_i \ne 1 \in H$ be the image of the $i$--th basis element $e_i \in \Z^k$ under the admissible homomorphism $\phi$.  As in \cite[Lemma~5.6]{COT04}, the tensor product of any chain complex with the contractible chain complex $\Q(H) \xrightarrow{(h_i-1)} \Q(H)$ of $S^1$ is again contractible.  Therefore the chain complexes $C_*(C_i \times S^1;\Q(H))$ and $C_*(S^1 \times S^1;\Q(H))$ are contractible, which is sufficient to prove the claim.  Therefore the Mayer-Vietoris sequence which calculates the homology of $\partial X_C$ as the union above implies the lemma.
\end{proof}

We are now ready to prove Theorem \ref{mainthm2}.

\begin{proof}[Proof of Theorem \ref{mainthm2}]
Let $C$ be a topological cobordism between links $L$ and $J$. Recall that we write
\[ l:=r(L,\phi)  \mbox{ and } j:=r(J,\phi),\]
and that we assume without loss of generality that $j\geq l$.  By Lemma \ref{lemma:homology_partial_X_C}, $$H_1(\partial X_C;\Q(H)) \cong H_1(Y_L;\Q(H)) \oplus H_1(Y_{J};\Q(H)).$$  It follows that
\[\dim_{\Q(H)}(H_1(\partial X_C;\Q(H))) = l+j.\]
Combining this with Corollary \ref{cor:chain_lifting_consequences} (2), which says that $$\dim_{\Q(H)} (H_1(X_C;\Q(H))) \leq l,$$ we deduce that
\[\dim_{\Q(H)}(\ker(H_1(\partial X_C;\Q(H))\to  H_1(X_C;\Q(H)))) \geq j.\]
By the long exact sequence in twisted homology of the pair $(X_C,\partial X_C)$:
\[H_2(X_C,\partial X_C;\Q(H)) \to H_1(\partial X_C;\Q(H))\to  H_1(X_C;\Q(H)),\]
this implies that
\[\dim_{\Q(H)}(\im(H_2(X_C,\partial X_C;\Q(H)) \to H_1(\partial X_C;\Q(H)))) \geq j,\]
so that
\[\dim_{\Q(H)}(H_2(X_C,\partial X_C;\Q(H))) \geq j.\]
By the universal coefficient theorem and Poincar\'{e}-Lefschetz duality,
\[H_2(X_C;\Q(H)) \cong H^2(X_C;\Q(H)) \cong \ol{H_2(X_C,\partial X_C;\Q(H))},\]
which then implies that
\[\dim_{\Q(H)}(H_2(X_C;\Q(H))) \geq j.\]

We have shown in Lemma \ref{lem:Qt_homology} that $H_i(X_C;\Q(H)) \cong 0$ for $i \neq 1,2$.  This implies that the Euler characteristic is given by: $$\chi^{\Q(H)}(X_C) = \dim_{\Q(H)}(H_2(X_C;\Q(H))) - \dim_{\Q(H)}(H_1(X_C;\Q(H))).$$
From the inequalities $$\dim_{\Q(H)}(H_2(X_C;\Q(H))) \geq j$$ and $$\dim_{\Q(H)} (H_1(X_C;\Q(H))) \leq l,$$ we obtain that
\[\chi^{\Q(H)}(X_C) \geq j-l.\]
Since the Euler characteristic can be calculated without taking homology, from the chain complex of the universal cover of $X_C$, by tensoring with $\Q$ or with $\Q(H)$ and taking the alternating sum of the ranks of the resulting chain groups, the Euler characteristic must be the same with either coefficient system.  So \[\chi^{\Q(H)}(X_C) = \chi^\Q(X_C) = 2g,\] by Corollary \ref{cor:euler_char}.  Therefore
\[2g \geq j-l,\]
which completes the proof of Theorem \ref{mainthm2}, and therefore also of Theorem \ref{mainthm}.
\end{proof}

\end{document}